\documentclass[12pt]{article}

\usepackage{verbatim}
\usepackage{enumerate}

\usepackage{color, float}

\usepackage{amssymb,amsthm,amsmath}
\usepackage{latexsym}
\usepackage{graphicx}
\usepackage{hyperref}
\usepackage{float}
\usepackage{algorithm}
\usepackage{algorithmic}
\usepackage{url}
\usepackage{float}
\usepackage{enumitem} 
\usepackage{subcaption}
\usepackage{url}
\usepackage{fullpage}

\usepackage{tikz}
\usetikzlibrary{arrows.meta, calc}
\usetikzlibrary{graphs, positioning, shapes.geometric}
\usetikzlibrary{decorations.markings}

\tikzset{
	mid arrow/.style={
		postaction={
			decorate,
			decoration={
				markings,
				mark=at position 0.8 with {
					\arrow{Latex[length=4mm, width=3mm]}
				}
			}
		}
	}
}

\newtheorem{theorem}{Theorem}

\newtheorem{cor}[theorem]{Corollary}
\newtheorem{obs}[theorem]{Observation}

\newtheorem{defi}[theorem]{Definition}
\newtheorem{prop}[theorem]{Proposition}

\title{On Link-irregular labelings of Graphs} 

\author{
	Alexander Bastien
	\qquad
	Omid Khormali\thanks{Corresponding author. Email: \texttt{ok16@evansville.edu}} \\
	\small Department of Mathematics\\[-0.8ex]
	\small University of Evansville\\[-0.8ex]
	\small Evansville, Indiana 47722, USA.\\
	\small \texttt{ab995@evansville.edu}\\
	\small \texttt{ok16@evansville.edu}
}

\begin{document}

	\maketitle
	
	\begin{abstract}
		We introduce the concept of link-irregular labelings for graphs, extending the notion of link-irregular graphs through edge labeling with positive integers. A labeling is link-irregular if every vertex has a uniquely labeled subgraph induced by its neighbors. We establish necessary and sufficient conditions for the existence of such labelings and define the link-irregular labeling number $\eta(G)$ as the minimum number of distinct labels required. Our main results include necessary and sufficient conditions for the existence of link-irregular labelings. We show that certain families of graphs, such as bipartite graphs, trees, cycles, hypercubes, and complete multipartite graphs, do not admit link-irregular labelings, while complete graphs and wheel graphs do. Specifically, we prove that $\eta(K_n) = 2$ for $n \geq 6$ and $\eta(K_n) = 3$ for $n \in \{3,4,5\}$. For wheel graphs $W_n$, we establish that $\eta(W_n) \approx \sqrt{2n}$ asymptotically. Finally, we prove that for every positive integer $n$, there exists a graph with a link-irregular labeling number exactly $n$, and provide several results on graph operations that preserve labeling numbers.
		
	\end{abstract}
	

	\section{Introduction}
	
	Let $G$ be a graph with vertex set $V(G)$ and edge set $E(G)$. We follow standard terminology and notation in graph theory as presented in West's textbook \cite{west}. The number of vertices and edges in a graph $G$ are denoted by $n(G)$ and $e(G)$, respectively. The number of edges connected to vertex $u \in V(G)$ is the degree of a vertex $u$, denoted $d(u)$. And the set of degrees of the vertices of graph $G$ is denoted by $D(G)$.
	The join $G \lor H$ is the graph formed by taking both graphs and adding all edges between vertices of $G$ and vertices of $H$.
	We use $G + H$ to denote the disjoint union of $G$ and $H$, meaning the graph consisting of all vertices and edges of $G$ and $H$ with no additional edges and no shared vertices. A graph is irregular if no two vertices in the graph have the same degree. It is known that no such graph exists because there is no simple graph in which all vertex degrees are distinct. Ali, Chartrand, and Zhang~\cite{akbar_book} introduced the notions of the \emph{link} of a vertex, \emph{link-regular} graphs, and \emph{link-irregular} graphs. The \emph{link} $L(v)$ of a vertex $v$ in a graph $G$ is defined as the subgraph induced by the neighborhood of $v$; that is, $L(v) = G[N(v)]$. A graph $G$ is called \emph{link-regular} if all of its vertices have isomorphic links—that is, $L(u) \cong L(v)$ for all $u, v \in V(G)$. Conversely, $G$ is \emph{link-irregular} if every pair of distinct vertices has non-isomorphic links; that is, $L(u) \ncong L(v)$ for all distinct $u, v \in V(G)$. One of their results is that there exists a link-irregular graph of order $n$ if and only if $n\geq 6$. The authors provided the unique link-irregular graph with 6 vertices in \cite{akbar_paper}, and it is
	
	\begin{figure}[htbp]
		\centering
		\includegraphics[scale = .3]{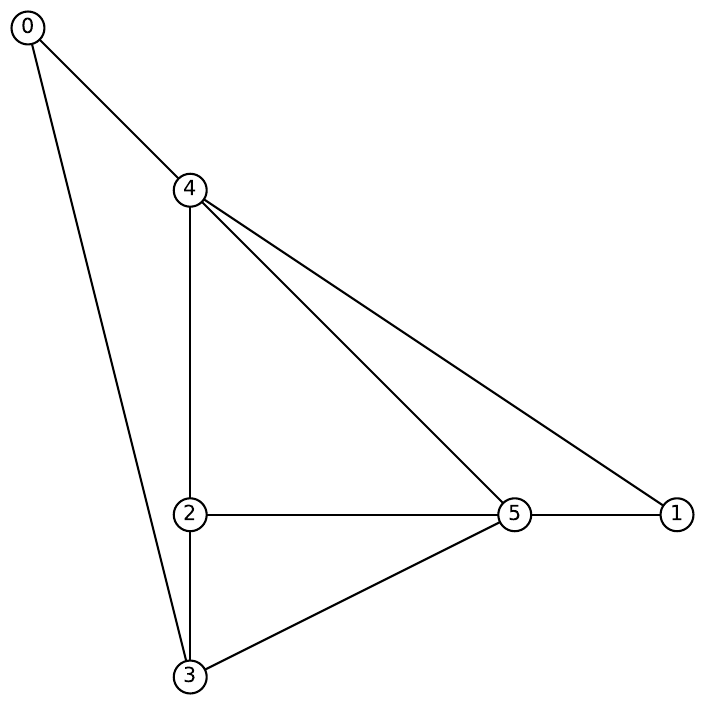}
		\caption{The unique link-irregular graph with 6 vertices \cite{akbar_paper}.}
		\label{fig:unique-6}
	\end{figure}
	
	Recently, Bastien and Khormali in \cite{alexander} revealed some surprising results: there are no bipartite or regular link-irregular graphs for small orders, and only finitely many link-irregular graphs are planar. Also, they disprove a conjecture that no regular link-irregular graphs exist, through explicit and probabilistic constructions. Link-irregularity is defined for simple graphs, and a natural generalization arises by extending the definition to loopless multigraphs where multiple edges are allowed between any pair of vertices. Instead of working directly with loppless multigraphs, we consider the edge-labeling of underlying simple graphs, where each edge is assigned a positive integer label indicating its multiplicity. This labeling helps us to work with simple graphs while still having the complexity of loopless multigraphs. So, for a simple graph in which each edge is labeled with a positive integer, we introduce a labeled link-irregular graph so that each vertex has a uniquely labeled subgraph induced by its neighbors. By this extension, we are able to study the structural uniqueness in cases where the simple graph alone is not link-irregular, but edge labeling can create link distinctions. Therefore, throughout the remainder of this paper, we consider only simple graphs that have no loops or multiple edges, and the labeling of edges is given by positive integers. \\
	
	
	Let $G$ be a graph and let $l: E(G) \to \mathbb{Z}^+$ be an edge-labeling of $G$, assigning a positive integer to each edge. For a vertex $v \in V(G)$, the \emph{labeled link} of $v$, denoted $L_l(v)$, is the induced subgraph of $G$ on the neighborhood $N(v)$, with edge labels coming from labeling $l$.
	
	We say that the labeling $l$ is \emph{link-irregular} if for all distinct vertices $u, v \in V(G)$, the labeled links are pairwise non-isomorphic:
	\[
	L_l(u) \ncong L_l(v) \quad \text{for all } u \ne v.
	\]
	
	A graph $G$ together with a link-irregular labeling $l$ is called a \emph{labeled link-irregular graph}.\\
	
	
	In the following, we investigate the structural properties of the labeled link-irregular graphs.
	
	\section{Results}
	
	The following result mentions the relation between the link and the labeled link of the vertices in a graph.

	\begin{theorem}\label{thm1}
		Let $ G$ be a graph. Then $G$ has a link-irregular labeling if and only if given any distinct $ x, y \in V(G) $, either $ L(x) \not\cong L(y) $, or  $ E(L(x)) \neq E(L(y)) $.
	\end{theorem}
	
	\begin{proof}
		
		Suppose the graph $G$ has a link-irregular labeling $l$. If $G$ is also a link-irregular graph, then the result follows immediately. Now suppose, for contradiction, that $G$ is not a link-irregular graph. Then there exist distinct vertices $x, y \in V(G)$ such that $L(x) \cong L(y)$ as unlabeled graphs. However, since $l$ is a link-irregular labeling, we have $L_l(x) \ncong L_l(y)$. But this is only possible if $E(L(x)) \ne E(L(y))$, because if $L(x) \cong L(y)$ and $E(L(x)) = E(L(y))$, then the labeled links would necessarily be isomorphic, contradicting the assumption that $L_l(x) \ncong L_l(y)$. Therefore, we conclude that $E(L(x)) \ne E(L(y))$ in $G$.\\

		For the converse, suppose $G$ is an unlabeled graph such that for all distinct $x, y \in V(G)$, either $L(x) \not\cong L(y)$ or $E(L(x)) \neq E(L(y))$. Create a labeling $l$ of the edges of $G$ by assigning a distinct positive integer to each edge. Then, for any two distinct vertices $x, y \in V(G)$, if $L(x) \not\cong L(y)$ as unlabeled graphs, then certainly $L_l(x) \ncong L_l(y)$. Otherwise, suppose $L(x) \cong L(y)$. Then $E(L(x)) \neq E(L(y))$, so there exists an edge $e \in E(L(x))$ such that $e \notin E(L(y))$. Since we assigned distinct labels to each edge in $G$, it follows that $e$ with its label appears in $L_l(x)$ but not in $L_l(y)$, and hence $L_l(x) \ncong L_l(y)$. Therefore, this is a link-irregular labeling.
	\end{proof}
	
	\begin{defi}
		The \emph{active neighborhood} of a vertex $x$ in a graph $G$, denoted $N_a(x)$, is the set of neighbors of $x$ that are adjacent to at least one other neighbor of $x$, that is, neighbors that are not isolated in the link $L(x)$. Formally,
		\[
		N_a(x) = \{ u \in N(x) : \deg_{G[N(x)]}(u) \ge 1 \}.
		\]
	\end{defi}
	
	The following result is a direct consequence of Theorem \ref{thm1} and the definition of active neighborhood.
	
	\begin{cor}\label{cor1}
		A graph $G$ has a link-irregular labeling if and only if for all distinct $x, y \in V(G)$, at least one of the following holds:
		\begin{enumerate}[label=(\alph*)]
			\item $L(x) \not\cong L(y)$, or
			\item both of the following:
			\begin{enumerate}[label=(\roman*)]
				\item $N_a(x) \ne N_a(y)$, and
				\item For each integer $n \ge 0$, at most one vertex $x \in V(G)$ satisfies $L(x) \cong \bar{K}_n$.
			\end{enumerate}
		\end{enumerate}
	\end{cor}

	\begin{proof}
		This is an immediate consequence of Theorem~\ref{thm1}, which states that $G$ is labeled link-irregular if and only if for all distinct $x, y \in V(G)$, either $L(x) \not\cong L(y)$ or $E(L(x)) \ne E(L(y))$. Condition (a) directly comes from the theorem. Now suppose $L(x) \cong L(y)$. Then, by the theorem, we must have $E(L(x)) \ne E(L(y))$. This occurs in one of two ways:
		\begin{itemize}
			\item If $L(x) \cong L(y) \cong \bar{K}_n$, then $E(L(x)) = \emptyset$, and the only way $E(L(x)) \ne E(L(y))$ is that $E(L(x))$ is empty set for at most one vertex $x$ in $G$ of degree $n$ that establishes condition (c).
			\item If $L(x)$ and $L(y)$ are non-empty but isomorphic, then any difference in their edge sets implies a difference in the neighborhoods of $x$ and $y$. Then, $N_a(x) \ne N_a(y)$, which gives condition (b).
		\end{itemize}
		
		Note that the original condition $E(L(x)) \ne E(L(y))$ is equivalent to the combined conditions (b) and (c), which completes the proof.
	\end{proof}
	
	The following observation is similar to Corollary \ref{cor1}, but states only the necessary condition. 
	
	\begin{obs}
		If a graph $G$ admits a link-irregular labeling, then:
		\begin{enumerate}
			\item $N(x) \ne N(y)$ for all distinct $x, y \in V(G)$, and
			\item For each $n \ge 0$, at most one vertex $x$ satisfies $L(x) \cong \bar{K}_n$.
		\end{enumerate}
	\end{obs}
	These conditions are necessary but not sufficient for the existence of a link-irregular labeling. To illustrate this, consider a graph $G$ in which two vertices $x$ and $y$ have isomorphic links of the form $K_2 + 2K_1$ (a disjoint union of an edge and two isolated vertices), and $G[N(x) \cap N(y)] = K_2$. In this case, although $N(x) \ne N(y)$ and at most one vertex may have an empty link, we have $L_l(x) \ncong L_l(y)$, and $G$ does not admit a link-irregular labeling, despite satisfying conditions (1) and (2).\ \\

	The following observation identifies several graph families that do not admit the link-irregular labeling conditions.
	
	
	\begin{obs}
		The following families of graphs do not admit a link-irregular labeling: all bipartite graphs, trees, cycles $C_n$ for $n \geq 4$, hypercubes $Q_n$ for $n \geq 2$, and complete multipartite graphs with at least one part containing two or more vertices.
	\end{obs}

	\begin{proof}
		Each of these graph families violates at least one of the necessary conditions stated in Corollary~\ref{cor1}.\\
		Cycles $C_n$ and hypercubes $Q_n$ are all regular graphs in which every vertex has an identical neighborhood structure in them. Consequently, for all $x \in V(G)$, the links $L(x)$ are pairwise isomorphic with identical edge sets. Therefore, no edge labeling can distinguish their labeled links.\\
		Bipartite graphs have the property that for every vertex $x$, the link $L(x)$ is edgeless. And multiple vertices in the bipartite graphs have the same degree since the set of degrees $D(G)$ has size less than $n$. This leads to multiple identical edgeless links, violating the uniqueness condition for link-irregular labeling. Trees are a subclass of bipartite graphs, and the above argument applies.\\ 
		If two vertices $x, y$ of a complete multipartite graph lie in the same part, then they have the same neighborhood, hence $L(x)\cong L(y)$ regardless of any possible edge labeling.
	\end{proof}
	
	We use $W_n$ as a wheel graph, which is a graph formed by connecting a single universal vertex to all vertices of a cycle $C_n$. Note that $W_4$ does not admit the link-irregular labeling. It can be seen that two non-adjacent vertices on the circle have isomorphic labeled links. \\
	
	The following observation identifies several graph families that admit the link-irregular labeling conditions. 
	
	\begin{obs}
		The following families of graphs admit a link-irregular labeling:
		\begin{itemize}
			\item Complete graphs $K_n$, $n\geq 3$,
			\item Wheels $W_n$, $n=3$, and $n\geq 5$.
		\end{itemize}
	\end{obs}
	\begin{proof}
		When we assign a distinct label to each edge of the graphs in both families, all labeled links are pairwise non-isomorphic. Hence, the desired result is concluded.
	\end{proof}
	
	Motivated by the above examples, we define a new graph invariant that quantifies the minimum number of labels required for a graph to admit a link-irregular labeling.
	
	\begin{defi}
		Let $G$ be a graph. The link-irregular labeling number of $G$, denoted by
		$\eta(G)$, is the minimum number of edge labels required to assign to $E(G)$ such that the resulting labeled graph is labeled link-irregular 
		. If no such labeling exists, we define $\eta(G) = \infty$.
	\end{defi}
	
	\begin{obs}
		A graph $G$ is link-irregular if and only if $\eta(G)=1$.
	\end{obs}
	
	In the following, we find the link-irregular labeling number of the complete graphs and Wheels.
	
	\begin{theorem}
		Let $ i_j \ge 5 $ for $ j \ge 3 $, and define the intervals for $ i_j $ based on the parity of $ j $ as follows:
		\[
		5 \leq i_3 \leq \binom{3+1}{2} < i_4 \leq \binom{4+1}{2} - \frac{4}{2} <  \cdots < i_j \leq  
		\begin{cases}
			\binom{j+1}{2} & \text{if } j \text{ is odd}, \\
			\binom{j+1}{2} - \frac{j}{2}  & \text{if } j \text{ is even}.
		\end{cases} < \cdots
		\]
		Then for each such $ i_j \neq 6,8$, we have $ \eta(W_{i_j}) = j $. Moreover, for large enough $n$, $\eta(W_{n}) \approx \sqrt{2n}$. In addition,  $\eta(W_3) = 3$, $\eta(W_6) = 5$, $\eta(W_8) = 5$ and $\eta(W_4) = \infty$.
		
		
	\end{theorem}
	\begin{proof}
		In the wheel graph $W_n$, the link of the universal vertex is $C_n$, which is unique. All other vertices lie on the cycle and have links isomorphic to a path of length two, where the central vertex of the path is the universal vertex. Thus, to ensure labeled link-irregularity, the labels of the edges from the universal vertex to the cycle vertices must yield non-isomorphic labeled links.
		

		Let $A$ denote the set of $n$ edges incident to the universal vertex in the wheel graph $W_n$. These edges connect the universal vertex to each of the $n$ vertices on the cycle. Each such edge $e \in A$ is involved in exactly two vertex links: the link of one cycle vertex adjacent to $e$ and the link of the cycle vertex which is in distance 2 in the clockwise or counterclockwise direction on the cycle. Therefore, every link of a cycle vertex involves two edges from the set $A$. When we assign labels to the edges in $A$ using $r$ distinct labels, each cycle vertex's link corresponds to a multiset of two labels. Since there are $n$ such links and each labeled link is formed from a pair of labels, we obtain $n$ such labeled multisets that should be non-isomorphic.
		
		Now consider $K_r^*$, the complete graph on $r$ labeled vertices with a loop at each vertex (but no multiple edges). A labeling of $A$ corresponds to a closed trail in $K_r^*$: each labeled link corresponds to an edge in the trail connecting the two labels used in that link. Since each edge in $A$ belongs to exactly two links, and links must be distinct, this process generates a trail that covers the edges of $A$.

		\begin{figure}[H]
			\centering
			\begin{subfigure}[t]{0.48\textwidth}
				\centering
				\begin{tikzpicture}[scale=0.9, every node/.style={circle, draw, minimum size=8pt, inner sep=2pt}]
					\def\n{15}
					\def\radius{2.8cm}
					
					\foreach \i in {1,...,\n} {
						\coordinate (n\i) at ({360/\n*(\i-1)}:\radius);
						\draw (n\i) -- (0,0); 
					}
					
					\foreach \i [evaluate=\i as \next using {int(mod(\i,\n)+1)}] in {1,...,\n} {
						\draw (n\i) -- (n\next);
					}
					
					\filldraw[black] (0,0) circle (1pt);
					
					\filldraw[black] (n15) circle (1pt); 

					\node at ($(0,0)!.7!(n1)$)  [above=.5pt, draw=none, shape=rectangle] {1};
					\node at ($(0,0)!0.7!(n2)$) [above=.7pt, draw=none, shape=rectangle] {4};
					\node at ($(0,0)!0.7!(n3)$) [above=.7pt, draw=none, shape=rectangle] {4};
					\node at ($(0,0)!0.7!(n4)$) [left=.5pt, draw=none, shape=rectangle] {4};
					\node at ($(0,0)!0.7!(n5)$) [left=.5pt, draw=none, shape=rectangle] {2};
					\node at ($(0,0)!0.7!(n6)$) [right=.5pt, draw=none, shape=rectangle] {3};
					\node at ($(0,0)!0.7!(n7)$) [above=.5pt, draw=none, shape=rectangle] {5};
					\node at ($(0,0)!0.7!(n8)$) [above=.5pt, draw=none, shape=rectangle] {3};
					\node at ($(0,0)!0.7!(n9)$) [above=.5pt, draw=none, shape=rectangle] {3};
					\node at ($(0,0)!0.7!(n10)$) [above=.5pt, draw=none, shape=rectangle] {2};
					\node at ($(0,0)!0.7!(n11)$) [left=.5pt, draw=none, shape=rectangle] {1};
					\node at ($(0,0)!0.7!(n12)$) [left=.5pt, draw=none, shape=rectangle] {2};
					\node at ($(0,0)!0.7!(n13)$) [right=.5pt, draw=none, shape=rectangle] {5};
					\node at ($(0,0)!0.7!(n14)$) [right=.7pt, draw=none, shape=rectangle] {1};
					\node at ($(0,0)!0.7!(n15)$) [above=.5pt, draw=none, shape=rectangle] {5};
					
					\foreach \i [evaluate=\i as \j using {mod(\i,15)+1}] in {1,...,15} {
						\draw (n\i) -- (n\j);
					}
					
					\coordinate (mid1) at ($(n1)!0.5!(n2)$) ;
					\coordinate (mid1) at ($(n2)!0.5!(n3)$) ;
					\coordinate (mid1) at ($(n3)!0.5!(n4)$) ;
					\coordinate (mid1) at ($(n4)!0.5!(n5)$) ;
					\coordinate (mid1) at ($(n5)!0.5!(n6)$) ;
					\coordinate (mid1) at ($(n6)!0.5!(n7)$) ;
					\coordinate (mid1) at ($(n7)!0.5!(n8)$) ;
					\coordinate (mid1) at ($(n8)!0.5!(n9)$) ;
					\coordinate (mid1) at ($(n9)!0.5!(n10)$) ;
					\coordinate (mid1) at ($(n10)!0.5!(n11)$) ;
					\coordinate (mid1) at ($(n11)!0.5!(n12)$) ;
					\coordinate (mid1) at ($(n12)!0.5!(n13)$) ;
					\coordinate (mid1) at ($(n13)!0.5!(n14)$) ;
					\coordinate (mid1) at ($(n14)!0.5!(n15)$) ;
					\coordinate (mid1) at ($(n15)!0.5!(n1)$) ;

				\end{tikzpicture}
				\caption{$W_{15}$}
				\label{fig:W15}
			\end{subfigure}
			\hfill
			\begin{subfigure}[t]{0.48\textwidth}
				\centering
				\begin{tikzpicture}[scale=0.8, every node/.style={circle, draw, minimum size=8pt, inner sep=2pt}]
					\node (1) at (90:2cm)  {1};
					\node (2) at (162:2cm) {2};
					\node (3) at (234:2cm) {3};
					\node (4) at (306:2cm) {4};
					\node (5) at (18:2cm)  {5};

					\draw[mid arrow, bend left=15] (1) to (2);
					\draw[mid arrow, bend left=15] (2) to (3);
					\draw[mid arrow, bend left=15] (3) to (4);
					\draw[mid arrow, bend left=15] (4) to (5);
					\draw[mid arrow, bend left=15] (5) to (1);
					\draw[mid arrow, bend left=15] (1) to (3);
					\draw[mid arrow, bend left=15] (3) to (5);
					\draw[mid arrow, bend left=15] (5) to (2);
					\draw[mid arrow, bend left=15] (2) to (4);
					\draw[mid arrow, bend left=15] (4) to (1);
					
					\foreach \i/\inangle/\outangle in {
						1/60/120,
						2/150/210,
						3/210/270,
						4/300/360,
						5/30/90
					} {
						\draw[mid arrow, looseness=15, in=\inangle, out=\outangle] (\i) to (\i);
					}
					
				\end{tikzpicture}
				\caption{$K_5^*$}
				\label{fig:K5star}
			\end{subfigure}
			
			\caption{A closed directed trail with arcs 11, 12, 22, 23, 33, 34, 44, 45, 55, 51, 13, 35, 52, 24, 41 in $K_5^*$ corresponds to a link-irregular labeling of the wheel graph $W_{15}$. The labeling begins at the vertex marked in black on the cycle. Each subsequent label is assigned to a vertex located two positions clockwise from the previous one, continuing this pattern until all cycle vertices are labeled.}
			\label{fig:W15andK5star}
		\end{figure}
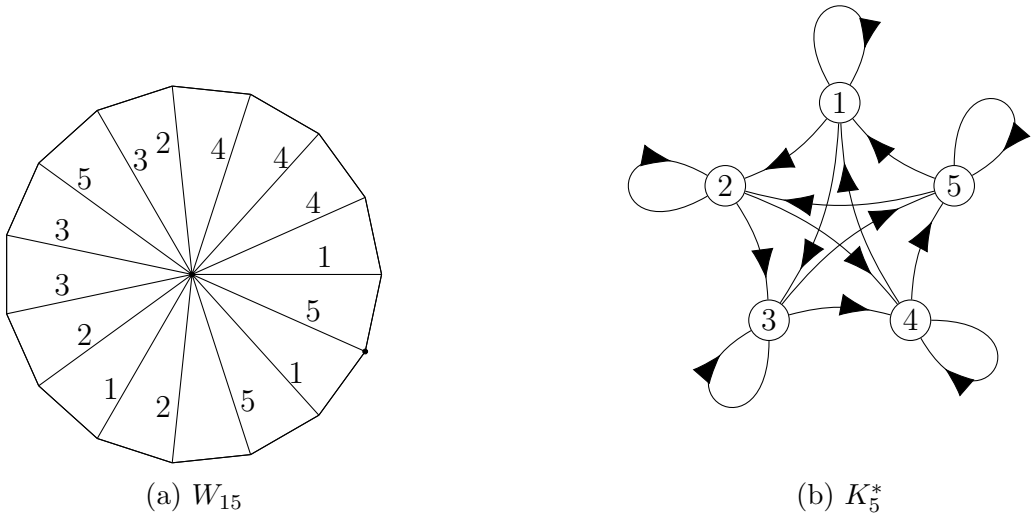
		
		When $n$ is odd, all edges in $A$ can be traversed in a single closed trail. When $n$ is even, they are split into two disjoint closed trails of length $n/2$. In either case, the labels used form the vertices of the trail(s), and the trail edges correspond to labeled links.
		
		Thus, finding the smallest number of labels $r$ such that $K_r^*$ contains a closed trail (or two edge-disjoint trails) covering $n$ edges is equivalent to determining $\eta(W_n)$.
		
		When $r$ is odd, $K_r^*$ has all vertices of even degree, and an Eulerian circuit (closed trail) exists that traverses all $\binom{r+1}{2}$ edges. When $r$ is even, the degrees are odd, and removing a perfect matching ($r/2$ edges) yields a graph with even degrees, supporting one or two equal-length edge-disjoint Eulerian circuits. Hence, for:
		\begin{center}
			even $r$: we need $n \le \binom{r+1}{2} - \frac{r}{2}$, \quad Odd $r$: we need $n \le \binom{r+1}{2}$.
		\end{center}
		This establishes the result.\\
		Also, based on the bounds for $n$ and $r$, we have one of the two inequalities $\binom{r}{2}<n\leq \frac{r^2}{2}$ and $\frac{r^2}{2}< n \leq \binom{r+1}{2}$. Solving each inequality for $r$ yields an approximation $r \approx \sqrt{2n}$ for large enough $n$. This confirms that $\eta(W_n) \approx \sqrt{2n}$ asymptotically.\\
		In addition, the following graph is labeled link-irregular $W_3$, and shows $\eta(W_3) = 3$. 
		\begin{figure}[H]
			\centering
			\begin{minipage}[t]{0.32\textwidth}
				\centering
				\includegraphics[scale = 0.15]{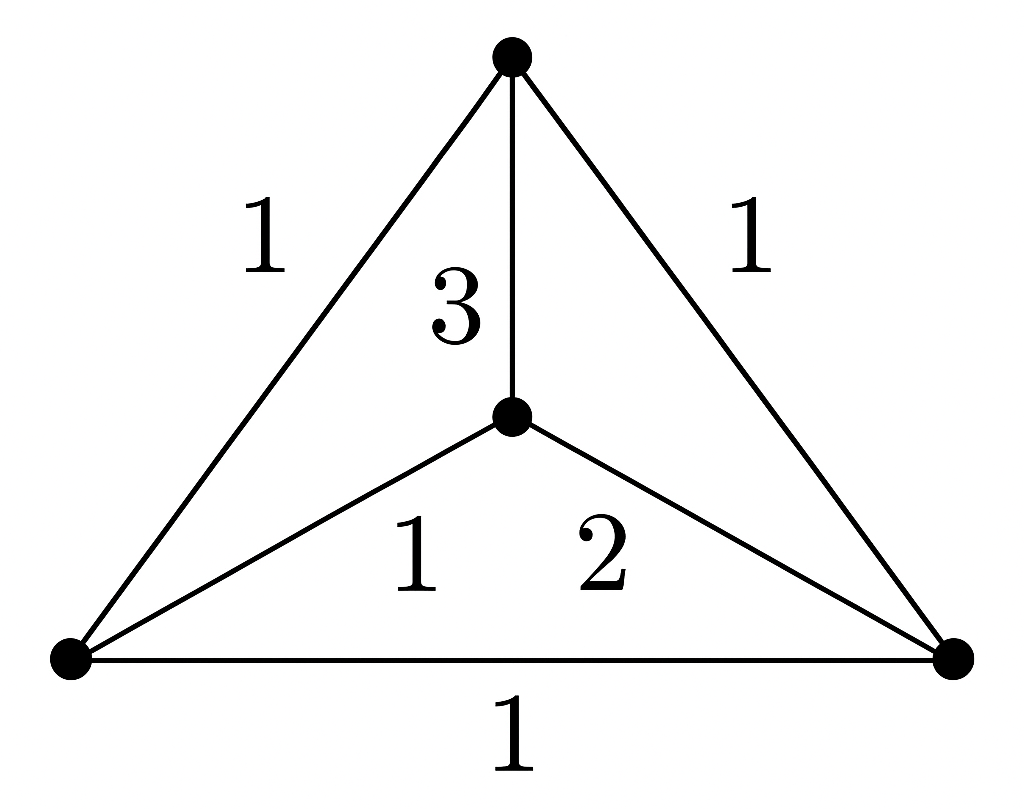}
				\caption*{(a) $W_3$}
			\end{minipage}
			\hfill
			\begin{minipage}[t]{0.32\textwidth}
				\centering
				\includegraphics[scale = 0.4]{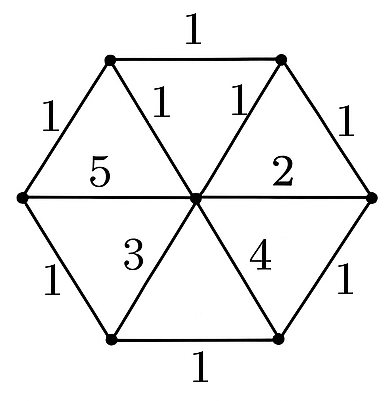}
				\caption*{(b) $W_6$}
			\end{minipage}
			\hfill
			\begin{minipage}[t]{0.32\textwidth}
				\centering
				\includegraphics[scale = 0.15]{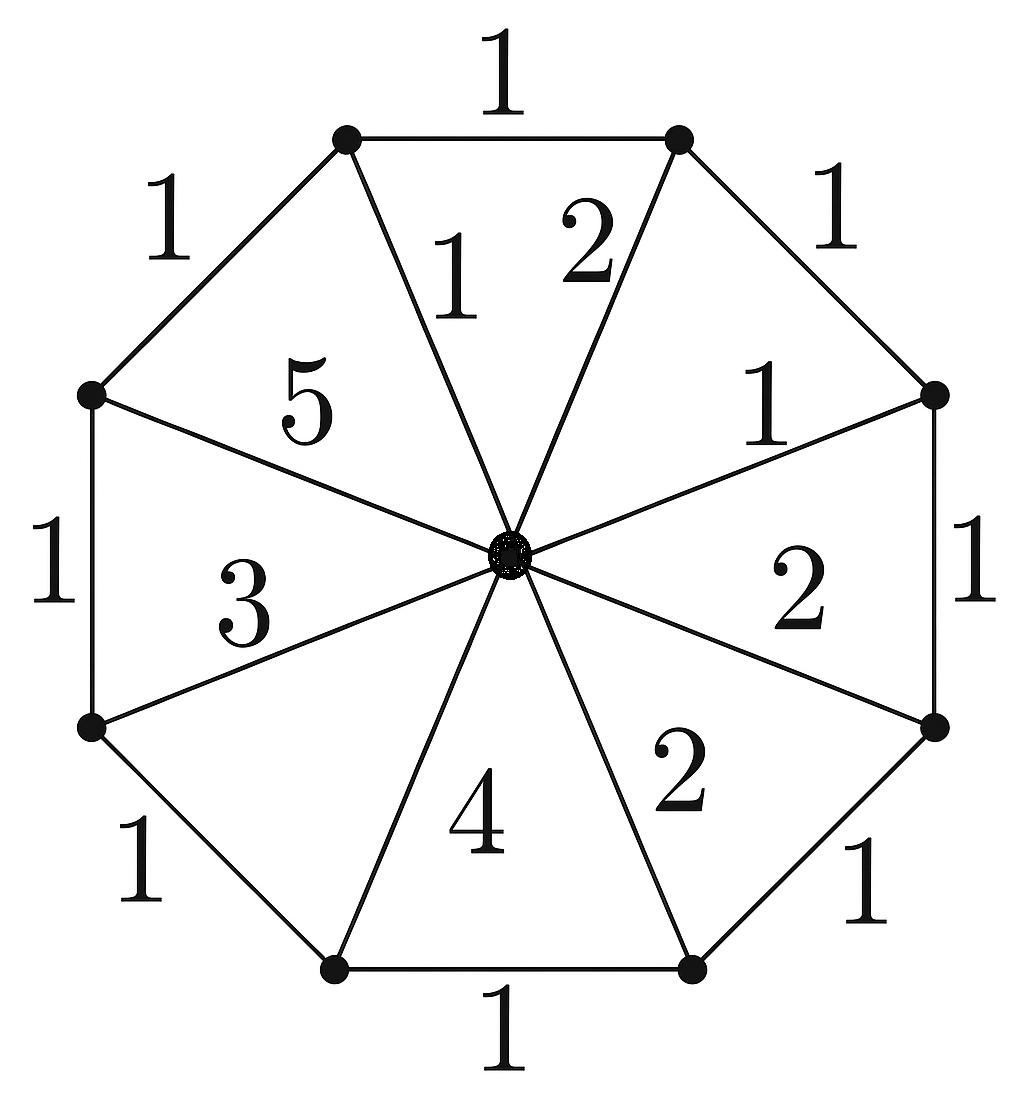}
				\caption*{(c) $W_8$}
			\end{minipage}
			\caption{A link-irregular labeling of $W_3$, $W_6$, and $W_8$}
			\label{fig:W3W6W8}
		\end{figure}
		Earlier, we mentioned that $W_4$ does not admit the link-irregular labeling. Then $\eta(W_4) = \infty$.
	\end{proof}

	We now aim to determine the link-irregular number of $K_n$ for $n \geq 3$. Observe that the labeled link of any vertex $v$ in an edge-labeled $K_n$ is the same as the graph $K_n-v$, which is a labeled $K_{n-1}$. Suppose that $K_n$ admits a link-irregular 2-labeling, and assume the edge labels used are ``red" and ``blue". Then all the key structural information of the labeled $K_n$ is contained in the graph $G_{red}$ that consists of all the vertices of $K_n$ and only the red edges. Note that $G_{red}$ may have vertices with degree zero. The labeled link of a vertex $v \in V(K_n)$ can be reconstructed from $G_{red} - v$ by adding any missing edges as blue. Consequently, we will be able to conclude that $K_n$ has a link-irregular 2-labeling if and only if there exists a graph $G$ on $n$ vertices such that for any two distinct vertices $u, v \in G$, $G-u \ncong G-v$. This motivates the following definition.
	
	
	\begin{defi}
		An unlabeled graph $G$ is said to be \textup{cut-irregular} if for all distinct vertices $u, v \in V(G)$, the vertex-deleted subgraphs $G - u$ and $G - v$ are non-isomorphic.
	\end{defi}

	We now explicitly state and prove the correspondence between the link-irregular 2-labelings of $K_n$ and cut-irregular graphs on $n$ vertices.
	
	
	\begin{prop}\label{prop:correspondence}
		Let $l$ be a link-irregular labeling of $K_n$ using the labels ``red" and ``blue". Then the subgraph $G_{\text{red}}$ consisting of all red edges in $K_n$ is a cut-irregular graph on $n$ vertices.
		
		Conversely, if $G$ is a cut-irregular graph on $n$ vertices, then assigning the label ``red'' to the edges of a subgraph of $K_n$ isomorphic to $G$, and labeling all remaining edges ``blue'', yields a link-irregular 2-labeling of $K_n$.
	\end{prop}

	\begin{proof}
		We first prove the first statement. Suppose $l$ is a link-irregular labeling of $K_n$. If $v$ is a vertex of $K_n$, $L(v) = K_n-v$. Suppose we have vertices $u$ and $v$ of $K_n$ such that $G_{red}-u \cong G_{red}-v$. Then adding all the blue edges to $G_{red}-u$ and $G_{red}-v$, we obtain that labeled $K_n-u$  is isomorphic to the labeled $K_n-v$, that is $L_l(u)\cong L_l(v)$. Since $l$ is a link-irregular labeling, this cannot hold, so $G_{red}-u \ncong G_{red}-v$, and $G_{red}$ is cut-irregular.\\
		Now we prove the second statement. Suppose $G$ is a cut-irregular graph. Label all edges of $G$ red and fill in in any non-edges of $G$ with a blue edge to obtain a labeling $l$ of $K_n$. If $L_l(u) \cong L_l(v)$ in $K_n$ then $K_n-u \cong K_n-v$. Then, by removing the blue edges, we obtain $G-u \cong G-v$. But, since $G$ is cut-irregular, this cannot be the case; hence, the labeling $l$ must be a link-irregular 2-labeling.
		
		\begin{figure}[H]
			\centering
			\begin{tikzpicture}[scale=1.1, every node/.style={circle, fill=black, inner sep=1pt}]
				\begin{scope}[xshift=0cm]
					\foreach \i in {1,...,6} {
						\node (p\i) at ({60*(\i-1)}:1cm) {};
					}
					
					\draw[blue, thick] (p1) -- (p2);
					\draw[blue, thick] (p1) -- (p5);
					\draw[blue, thick] (p1) -- (p4);
					\draw[blue, thick] (p1) -- (p3);
					\draw[blue, thick] (p2) -- (p6);
					\draw[blue, thick] (p2) -- (p5);
					\draw[blue, thick] (p2) -- (p4);
					\draw[blue, thick] (p4) -- (p6);
					
					\draw[red] (p1) -- (p6);
					\draw[red] (p2) -- (p3);
					\draw[red] (p3) -- (p6);
					\draw[red] (p3) -- (p5);
					\draw[red] (p3) -- (p4);
					\draw[red] (p4) -- (p5);
					\draw[red] (p5) -- (p6);

					\node[draw=none, fill=none] at (0,-1.5) {$K_5$};
				\end{scope}
				
				\begin{scope}[xshift=6cm]
					\foreach \i in {1,...,6} {
						\node (p\i) at ({60*(\i-1)}:1cm) {};
					}
					
					\draw[red] (p1) -- (p6);
					\draw[red] (p2) -- (p3);
					\draw[red] (p3) -- (p6);
					\draw[red] (p3) -- (p5);
					\draw[red] (p3) -- (p4);
					\draw[red] (p4) -- (p5);
					\draw[red] (p5) -- (p6);
					
					\node[fill=black] at (p1) {};
					\node[draw=none, fill=none] at (0,-1.3) {$G_{\text{red}}$};
					\node[draw=none, fill=none] at ($(p4)+(-0.3,0)$) {\small $v$};
				\end{scope}
				
				\draw[<->, thick] (2.5,0) -- (3.5,0);
				
				\begin{scope}[yshift=-3.5cm]
					\node (a) at (-.5,0) {};
					\node (b) at (.5,0) {};
					\node (c) at (.5,1) {};
					\node (d) at (-.5,1) {};
					\node (e) at (.5+1.732*.5,.5) {};

					\draw[red] (a) -- (b);
					\draw[red] (a) -- (d);
					\draw[red] (b) -- (e);
					\draw[red] (b) -- (d);
					\draw[red] (d) -- (c);
					
					\draw[blue] (e) -- (c);
					\draw[blue] (e) -- (d);
					\draw[blue] (e) -- (a);
					\draw[blue] (b) -- (c);
					\draw[blue] (a) -- (c);

					\node[draw=none, fill=none] at (0.75,-0.8) {labeled $L(v)$};
				\end{scope}
				
				\begin{scope}[xshift=6cm, yshift=-3.5cm]
					
					\node (a) at (-.5,0) {};
					\node (b) at (.5,0) {};
					\node (c) at (.5,1) {};
					\node (d) at (-.5,1) {};
					\node (e) at (.5+1.732*.5,.5) {};

					\draw[red] (a) -- (b);
					\draw[red] (a) -- (d);
					\draw[red] (b) -- (e);
					\draw[red] (b) -- (d);
					\draw[red] (d) -- (c);

					\node[draw=none, fill=none] at (0.75,-0.8) {$G_{\text{red}} - v$};
				\end{scope}
				
				\draw[<->, thick] (2.5,-3.2) -- (3.5,-3.2);
				
			\end{tikzpicture}
			\caption{Transformation between a labeled complete graph $K_6$ and its link graph $L(v)$, compared to a pencil-drawn $G_{\text{red}}$ and $G_{\text{red}} - v$.}
		\end{figure}
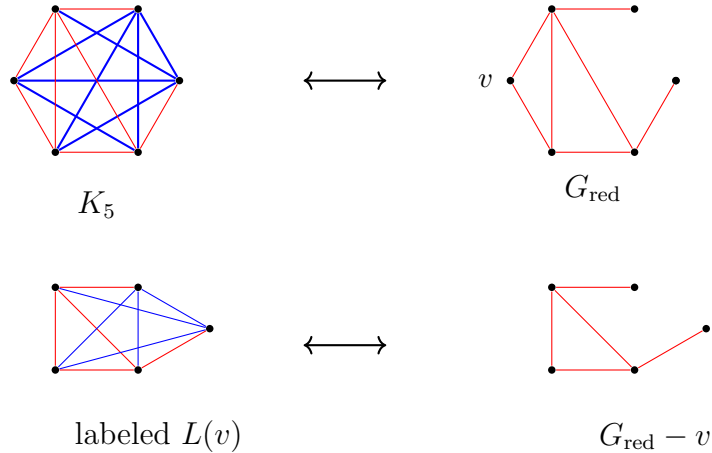      
	\end{proof}
	
	We have thus reduced the problem of determining whether $K_n$ admits a link-irregular 2-labeling to the problem of finding a cut-irregular graph on $n$ vertices. As we now show, such a graph exists if and only if $n \geq 6$. The proof is analogous to the one presented in~\cite{akbar_book}.

	\begin{theorem} \label{thm:existence}
		There exists a cut-irregular graph on $n$ vertices if and only if $n \geq 6$.
	\end{theorem}
	
	\begin{proof}
		That there is no cut-irregular graph on fewer than six vertices is readily verified by examining each graph on five or fewer vertices. Below is a cut-irregular graph on six vertices. Call this graph $G_6$ This is our base step. 
		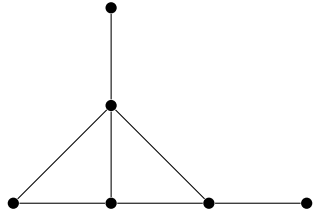
\begin{figure}[H]
			\centering
			\begin{tikzpicture}[scale=1.3, every node/.style={circle, fill=black, inner sep=1.5pt}]
				\node (A) at (0,0) {};
				\node (B) at (1,0) {};
				\node (C) at (2,0) {};
				\node (D) at (3,0) {};
				\node (E) at (1,1) {};
				\node (F) at (1,2) {};
				
				\draw (A) -- (B);
				\draw (B) -- (C);
				\draw (C) -- (D);
				\draw (A) -- (E);
				\draw (B) -- (E);
				\draw (C) -- (E);
				\draw (F) -- (E);
			\end{tikzpicture}
			\caption{A cut-irregular graph on 6 vertices.}
			\label{cut_irr}
		\end{figure}
		
		From this, we will construct an infinite set $\lbrace G_6, G_7, G_8, \cdots \rbrace$ such that $G_n$, $n \geq 6$, is a cut-irregular graph on $n$ vertices. For $n > 6$, we construct $G_n$ as follows
		\begin{itemize}
			\item If $n$ is odd, we put $G_n = G_{n-1} \lor K_1$. That is, $G_n$ is the join of $G_{n-1}$ and $K_1$. 
			\item If $n$ is even, we obtain $G_n$ by adding a new vertex and connecting it to a vertex of minimum degree in $G_{n-1}$. 
		\end{itemize}
		We claim that each $G_n$ is cut-irregular. We first notice that $G_6$ has no vertex of degree 5. Hence $G_7$ has only one vertex of degree 6. $G_7$ also has no vertex of degree 1 since $G_6$ has no isolated vertices and joining $K_1$ to $G_6$ increases the degree of each vertex by 1. Then $G_8$ has only one vertex of degree 1 and no vertex of degree 7. More generally, we have:
		\begin{itemize}
			\item If $n$ is odd, then $G_n$ has exactly one vertex of degree $n-1$ and no vertex of degree 1.
			\item If $n$ is even, then $G_n$ has exactly one vertex of degree 1 and no vertex of degree $n-1$.
		\end{itemize}
		Suppose that $G_{n-1}$ is cut-irregular for $n-1 \geq 6$. Then we examine two cases.\\\\
		\textit{Case 1. $n$ is odd.} Then, $n-1$ is even, and by construction,  $G_{n-1}$ has exactly one vertex of degree 1 and no vertex of degree $n-1$. Then, to obtain $G_n$, we join a new vertex $w$ to all the vertices of $G_{n-1}$. Suppose for contradiction that $G_n - u \cong G_n - v$ for some $u, v \in V(G_n)$. First, consider the case where both $u$ and $v$ are in $G_{n-1}$. Then, We have $G_n - u = (G_{n-1}\lor K_1) - u \cong (G_{n-1}-u) \lor K_1$. Likewise, $G_n - v \cong (G_{n-1}-v)\lor K_1$. Therefore, $G_n - u \cong G_n - v$ implies $(G_{n-1} - u) \lor K_1 \cong (G_{n-1} - v) \lor K_1$. Removing the common $K_1$ from both sides, we get $G_{n-1} - u \cong G_{n-1} - v$, contradicting the assumption that $G_{n-1}$ is cut-irregular. Hence, $G_n - u \not\cong G_n - v$ when $u, v \in V(G_{n-1})$.\\
		Now consider the case where $v = w$. Observe that if $G_n - u \cong G_n - w$, then $u$ and $w$ must have the same degree in $G_n$; otherwise, the graphs would have different numbers of edges and could not be isomorphic. Since $w$ is the only vertex of degree $n - 1$ in $G_n$, no other vertex can have the same degree, so $G_n - u \not\cong G_n - w$ for any $u \in V(G_{n-1})$. Therefore, in all cases, $G_n - u \not\cong G_n - v$ for distinct $u, v \in V(G_n)$, and thus $G_n$ is cut-irregular in this case.\\

		\textit{Case 2. $n$ is even.} Then $n - 1$ is odd, and by construction, $G_{n-1}$ has exactly one vertex of degree $n - 1$ and no vertex of degree 1. We obtain $G_n$ by adding a new vertex $w$ and connecting it to a vertex $z$ of minimum degree in $G_{n-1}$, making $w$ the only vertex of degree 1 in $G_n$. Let $u$ and $v$ be arbitrary vertices in $G_{n-1} \setminus \{z\}$. We aim to show that $G_n - u \ncong G_n - v$. First, observe that $w$ remains the unique vertex of degree 1 in both $G_n - u$ and $G_n - v$, so removing $w$ from both graphs gives
		\[
		(G_n - u) - w \cong G_{n-1} - u \quad \text{and} \quad (G_n - v) - w \cong G_{n-1} - v.
		\]
		Therefore, if $G_n - u \cong G_n - v$, then $(G_n - u) - w \cong (G_n - v) - w$, and $G_{n-1} - u \cong G_{n-1} - v$, which contradicts the assumption that $G_{n-1}$ is cut-irregular. Next, we claim that $G_n - z \ncong G_n - w$. Indeed, $G_n - z$ contains an isolated vertex (namely $w$), while $G_n - w$ is connected. Similarly, $G_n - w \ncong G_n - u$ for any $u \in V(G_{n-1})$ since $w$ is the only vertex of degree 1, and $G_n - u$ has no vertex of degree 1. Finally, to show that $G_n - z \ncong G_n - u$, note that $G_n - z$ contains an isolated vertex ($w$), while $G_n - u$ does not, because $w$ is only adjacent to $z$ and $z \ne u$. Hence, in all cases, $G_n - u \ncong G_n - v$ for distinct $u, v \in V(G_n)$, and so $G_n$ is cut-irregular in this case as well. This completes the proof.

		\begin{figure}[H]
			\centering
			\begin{tikzpicture}[scale=1.1, every node/.style={circle, fill=black, inner sep=1.2pt}]
				
				\begin{scope}[xshift=1cm]
					\foreach \i in {1,...,6} {
						\node (b\i) at ({60*(\i-1)+90}:1) {};
					}
					
					
					\draw (b1) -- (b2);
					\draw (b2) -- (b3) -- (b4) -- (b5) -- (b2);
					\draw (b5) -- (b6);
					\draw (b2) -- (b4);
					
					\node[draw=none, fill=none] at (0,-1.5) {$G_6$};
				\end{scope}
				
				\begin{scope}[xshift=5cm]
					\foreach \i in {1,...,7} {
						\node (b\i) at ({360/7*(\i-1)+90}:1) {};
					}
					
					\draw (b1) -- (b2);
					\draw (b2) -- (b3) -- (b4) -- (b5) -- (b2);
					\draw (b5) -- (b6);
					\draw (b2) -- (b4);
					\draw (b7) -- (b1);
					\draw (b7) -- (b2);
					\draw (b7) -- (b3);
					\draw (b7) -- (b4);
					\draw (b7) -- (b5);
					\draw (b7) -- (b6);
					
					\node[draw=none, fill=none] at (0,-1.5) {$G_7$};
				\end{scope}
				
				\begin{scope}[xshift=9cm]
					\foreach \i in {1,...,7} {
						\node (b\i) at ({360/7*(\i-1)+90}:1) {};
					}
					\node (c8) at (0,2) {};  
					
					\draw (b1) -- (b2);
					\draw (b2) -- (b3) -- (b4) -- (b5) -- (b2);
					\draw (b5) -- (b6);
					\draw (b2) -- (b4);
					\draw (b7) -- (b1);
					\draw (b7) -- (b2);
					\draw (b7) -- (b3);
					\draw (b7) -- (b4);
					\draw (b7) -- (b5);
					\draw (b7) -- (b6);
					\draw (b1) -- (c8);
					
					\node[draw=none, fill=none] at (0,-1.5) {$G_8$};
				\end{scope}
				
			\end{tikzpicture}
			\caption{The graphs $G_6, G_7,$ and $G_8$.}
			\label{G6G7G8}
		\end{figure}
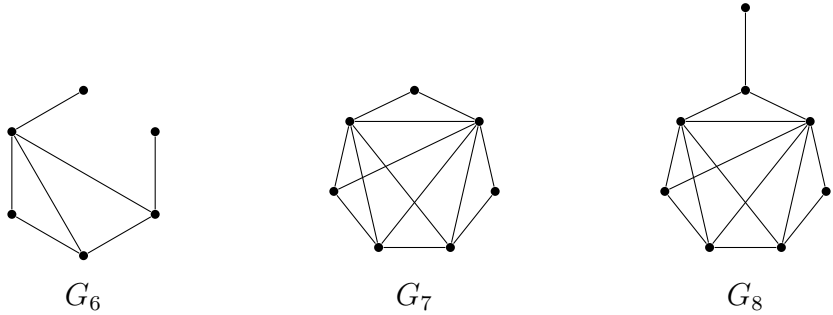 
	\end{proof}
	
	
	
	
	We now have the following result.
	
	\begin{theorem}
		The link-irregular numbers of the $K_n$ are as follows
		\begin{enumerate}
			\item $\eta(K_1) = 0$ and $\eta(K_2) = \infty$
			\item $\eta(K_n) = 3$ if $n = 3, 4, 5$
			\item $\eta(K_n) = 2$ if $n \geq 6$
		\end{enumerate}
	\end{theorem}
	
	\begin{proof}
		It is straightforward to verify that $\eta(K_1) = 0$ and that $\eta(K_2) = \infty$.
		That $\eta(K_n) = 2$ if $n \geq 6$ follows directly from Theorem~\ref{thm:existence} and Proposition~\ref{prop:correspondence}. \\It remains to show that $\eta(K_n) = 3$ if $n = 3, 4, 5$.
		By Theorem~\ref{thm:existence}, there is no cut-irregular graph on 3, 4, or 5 vertices, so $\eta(K_n) \geq 3$ for $n=3,4,5$. In Figure \ref{complete_graphs}, the link-irregular 3-labelings of $K_3, K_4,$ and $K_5$ are shown. So, $\eta(K_n) \leq 3$. Hence, $\eta(K_n)=3$ for $n=3,4,5$.
		
		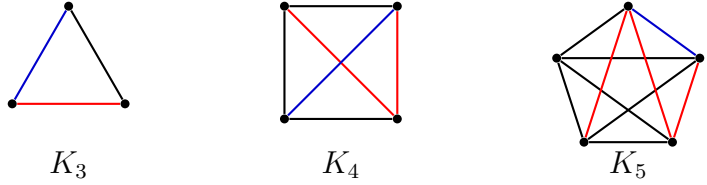
\begin{figure}[H]
			\centering
			\begin{tabular}{c@{\hspace{2cm}}c@{\hspace{2cm}}c}
				\begin{tikzpicture}[scale=1.5, baseline=(current bounding box.north)]
					\node[circle, fill=black, inner sep=1.2pt] (a) at (0,0) {};
					\node[circle, fill=black, inner sep=1.2pt] (b) at (1,0) {};
					\node[circle, fill=black, inner sep=1.2pt] (c) at (0.5,0.866) {};
					\draw[red, thick] (a) -- (b);
					\draw[black, thick] (b) -- (c);
					\draw[blue!80!black, thick] (c) -- (a);
				\end{tikzpicture}
				&
				\begin{tikzpicture}[scale=1.5, baseline=(current bounding box.north)]
					\node[circle, fill=black, inner sep=1.2pt] (a) at (0,0) {};
					\node[circle, fill=black, inner sep=1.2pt] (b) at (1,0) {};
					\node[circle, fill=black, inner sep=1.2pt] (c) at (1,1) {};
					\node[circle, fill=black, inner sep=1.2pt] (d) at (0,1) {};
					\draw[black, thick] (a) -- (b);
					\draw[black, thick] (a) -- (d);
					\draw[black, thick] (c) -- (d);
					\draw[red, thick] (b) -- (c);
					\draw[red, thick] (b) -- (d);
					\draw[blue!80!black, thick] (a) -- (c);
				\end{tikzpicture}
				&
				\begin{tikzpicture}[scale=1, baseline=(current bounding box.north)]
					\foreach \i in {1,...,5} {
						\node[circle, fill=black, inner sep=1.2pt] (p\i) at ({90+72*(\i-1)}:1cm) {};
					}
					\draw[black, thick] (p1) -- (p2);
					\draw[black, thick] (p2) -- (p3);
					\draw[black, thick] (p2) -- (p4);
					\draw[black, thick] (p2) -- (p5);
					\draw[black, thick] (p3) -- (p4);
					\draw[black, thick] (p3) -- (p5);
					\draw[red, thick] (p1) -- (p3);
					\draw[red, thick] (p1) -- (p4);
					\draw[red, thick] (p4) -- (p5);
					\draw[blue!80!black, thick] (p5) -- (p1);
				\end{tikzpicture}
				\\[2pt]
				$K_3$ & $K_4$ & $K_5$
			\end{tabular}
			\caption{Examples of complete graphs $K_3$, $K_4$, and $K_5$ with selected colored edges.}
			\label{complete_graphs}
		\end{figure}
	\end{proof}
	
	Note that $K_n$ is link-regular for all $n \geq 2$, as the link of every vertex is isomorphic to $K_{n-1}$. Thus, $K_n$ provides an example of a link-regular graph that admits a link-irregular labeling.
	

	
	
	
	\begin{obs}
		There exist link-regular graphs that also admit link-irregular labelings.
	\end{obs}
	
	This highlights that link-regularity and the existence of link-irregular labelings are distinct properties. While $K_n$ is link-regular and admits a link-irregular labeling, not all link-regular graphs share this property. For instance, cycles $C_n$ with $n \geq 4$ are link-irregular but do not admit any link-irregular labeling.
	This demonstrates that while link-irregularity generalizes the classical notion of irregularity, link-irregular labelings further expand the link-irregularity.\\
	
	At this point, we answer the question: for which values of $n$ does there exist a graph $H_n$ such that $\eta(H_n) = n$? Using the facts that $\eta(K_3) = 3$ and $\eta(G) = 1$ for any link-irregular graph $G$, we show that such a graph $H_n$ exists for every integer $n \geq 1$. Note that while we claim that such a graph exists, we are not claiming that there is only one such graph.
	
	\begin{theorem}
		For $n=1,2,3, \cdots$, there exists a graph $H_n$ such that $\eta(H_n)=n$.
	\end{theorem}
	
	\begin{proof}
		We construct the graph $H_n$ in three cases based on the congruence of $n$ modulo 3: $n \equiv 0 \pmod{3}$, $n \equiv 1 \pmod{3}$, and $n \equiv 2 \pmod{3}$. Throughout the proof, we denote by $G$ the unique link-irregular graph on six vertices in Figure \ref{fig:unique-6}. We now consider our three cases.\\
		
		\textit{Case 1. $n \equiv 0 \pmod{3}.$} In this case, we can write $n$ in the form $n=3k$. Then consider the graph $H_n=kK_3$. Each edge of $H_n$ is in one and only one link, end each link is a $K_2$. Thus assigning every edge of $H_n$ a unique label provides us with a link-irregular labeling of $H_n$ using $n=3k$ labels. If we use any fewer labels, then two edges must have the same label. Since each edge constitutes the link of one vertex, this means two vertices have isomorphic labeled-links, so a labeling with fewer labels is not link-irregular. Hence $\eta(H_n)=n$.\\
		
		
		
		\textit{Case 2. $n \equiv 1 \pmod{3}.$} In this case, we write $n = 3k + 1$. Consider the graph $H_n = kK_3 + G$, which has $3k + 6$ vertices. This graph includes exactly $3k + 1 = n$ vertices whose links are isomorphic to $K_2$. Since each such link must receive a distinct label, we have $\eta(H_n) \geq n$. To construct a link-irregular labeling with $n$ labels, assign a unique label to each edge of the $kK_3$ component, and assign a single, unused label to all edges of $G$. Note that this is valid because $G$ is link-irregular, and every vertex outside of $G$ has link $K_2$. Although one vertex of $G$ also has $K_2$ as its link, labeling the edges of $G$ with a common, distinct label ensures that all $K_2$ links remain distinguishable. Therefore, we conclude that $\eta(H_n) = n$.\\
		
		\textit{Case 3. $n \equiv 2 \pmod{3}.$} In this case, write $n = 3k + 2$. Let $H_n = kK_3 + 2G$, which has $3k + 12$ vertices. In this graph, exactly $3k + 2 = n$ vertices have $K_2$ as their link, so we have $\eta(H_n) \geq n$. To construct a link-irregular labeling using exactly $n$ labels, we assign a unique label to each edge of the $kK_3$ component. For the two copies of $G$, assign all edges in the first copy the same new label, and do likewise for the second copy using a different new label not used previously. Since $G$ is link-irregular and none of the links are isomorphic, this labeling ensures link-irregularity. Thus, we obtain a link-irregular labeling of $H_n$ using $n = 3k + 2$ labels, so $\eta(H_n) = n$. This completes the proof.
	\end{proof}
	
	In the following propositions, we investigate the link-irregular labeling number of graphs of the form $G \lor K_n$ under certain conditions.

	\begin{prop} \label{prop:expanding}
		Suppose $\Delta(G)<|G|-1$ and $\eta(K_n) \leq \eta(G) < \infty$ for some value $n$. Then $\eta(G \lor K_n) = \eta(G)$.
	\end{prop}
	
	\begin{proof}
		It is clear that $\eta(G \lor K_n) \geq \eta(G)$, so we only need to show that $\eta(G \lor K_n) \leq \eta(G)$. We do this by providing a link-irregular $\eta(G)$-labeling of $G \lor K_n$. Suppose we have a link-irregular labeling of $G$ using labels from $[\eta(G)]$ and using the same labeling-set, we may choose a link-irregular labeling for $K_n$ too. We define the labeling $l$ of $G \lor K_n$ as follows:\\
		Label $G$ and $K_n$ according to the labelings we have chosen for them.\\
		Label the edges between $G$ and $K_n$ all with the same label, say $1$.\\
		We claim this is a link-irregular labeling. Clearly no two vertices in $G$ have the same labeled link, and no two vertices in $K_n$ have the same link. Suppose $u \in V(G)$ and $v \in V(K_n)$. Then the link of $u$ has $n$ points of degree $|G|+n-2$ while the link of $v$ only has $n-1$ points of degree $|G|+n-2$ hence the links cannot be isomorphic.
	\end{proof}
	
	\begin{prop} \label{prop:reducing}
		Let $G$ be a graph such that $\eta(G) < \infty$, and let $H$ be the subgraph of $G$ induced by all vertices of degree $|G|-1$. If $3 \leq \eta(G-H) < \infty$ then $\eta(G) = \eta(G-H)$.
	\end{prop}
	
	\begin{proof} 
		In this proof, we make use of Proposition~\ref{prop:expanding}. If $H$ is empty, the result is trivial, so assume $|H|>0$. Since $H$ contains all vertices of degree $|G|-1$, $G-H$ has no vertices of degree $|G-H|-1$. By Proposition~\ref{prop:expanding}, $\eta((G-H)\lor K_{|H|}) = \eta(G-H)$, but $(G-H)\lor K_{|H|} \cong G$. Hence $\eta(G) = \eta(G-H)$.
	\end{proof}
	
	\begin{prop}
		Suppose $\eta(K_{n}) \leq \eta(G) < \infty$ for some value $n$, and $3 \leq \eta(G-H) < \infty$ where $H$ is the subgraph induced by the vertices of degree $|G|-1$ in $G$. Then $\eta(G \lor K_n) = \eta(G)$. 
	\end{prop}
	
	\begin{proof}
		If $G$ has no vertex of degree $|G|-1$, this is just Proposition~\ref{prop:expanding}. If $G$ does have a vertex of degree $|G|-1$, then we let $H$ be as in the proof of Proposition~\ref{prop:reducing}. We have $G\lor K_n \cong (G-H)\lor K_{n+|H|}$, so using both Proposition~\ref{prop:expanding} and Proposition~\ref{prop:reducing}, we obtain $\eta(G\lor K_n) = \eta((G-H) \lor K_{n+|H|}) = \eta(G-H) = \eta(G)$.
	\end{proof}

	\section*{Acknowledgments}
	This research was supported by the UExplore Undergraduate Research Program at the University of Evansville.

\end{document}